\documentclass[11pt,reqno]{amsart}
\usepackage{amssymb,amsmath,amsthm,amsfonts,latexsym,fullpage}

\newtheorem{theorem}{Theorem}[section]
\newtheorem{lemma}[theorem]{Lemma}

\theoremstyle{remark}
\newtheorem*{notation}{Notation}

\numberwithin{equation}{section}

\def \le{{\,\leqslant\,}}
\def \ge{{\,\geqslant\,}}

\begin{document}

\title[Sums of four squares of primes]
{On Sums of four squares of primes}

\author{Angel Kumchev}
\address{Department of Mathematics\\
         Towson University\\
         7800 York Road\\
         Towson, MD 21252\\
         U.S.A.}
\email{akumchev@towson.edu}

\author{Lilu Zhao}
\address{School of Mathematics, Hefei University of Technology,
Hefei, 230009, China} \email{zhaolilu@gmail.com}

\thanks{L. Zhao is supported by the National Natural Science Foundation of China (Grant No. 11401154).}
\subjclass[2010]{11P32.}
\keywords{Circle method, sieve method.}

\begin{abstract}
  Let $E(N)$ denote the number of positive integers $n \le N$, with $n \equiv 4 \pmod{24}$, which cannot be represented as the sum of four squares of primes. We establish that $E(N)\ll N^{11/32}$, thus improving on an earlier result of Harman and the first author, where the exponent $7/20$ appears in place of $11/32$.
\end{abstract}

\maketitle

\section{Introduction}

Let 
\[ 
  \mathcal A = \big\{ n \in \mathbb N: n \equiv 4 \pmod{24} \big\}.
\]
It is conjectured that every sufficiently large integer $n \in \mathcal A$ can be represented as the sum of four squares of primes. Since this conjecture appears to lie beyond the reach of present methods, several approximations to it have been studied. One of those recasts the question in terms of the set of possible exceptions. Let $\mathcal{E}$ denote the set of $n \in \mathcal{A}$ that have no representations as the sum of four squares of primes. Hua~\cite{Hua} was the first to prove that this exceptional set is ``thin.'' Write $E(N)$ for the cardinality of $\mathcal E \cap [1,N]$. Hua showed that
\begin{equation}\label{log}
  E(N)\ll N(\log N)^{-A}
\end{equation}
for some absolute constant $A > 0$. Later, Schwarz \cite{Schwarz} refined Hua's result and showed that the power of the logarithm in \eqref{log} can be chosen arbitrarily large. 

A couple of breakthroughs occurred at the cusp between the last and current centuries. First, Liu and Zhan \cite{LiuZhan} discovered a new technique for dealing with the major arcs in the application of the circle method. That was followed closely by a clever observation of Wooley \cite{Wool} that greatly improved some minor arc estimates.  Those ideas led to a series of improvements on \eqref{log} (see \cite{Liu, LiuLiu, LiuWooley, Wool}), culminating in the result of Liu, Wooley and Yu~\cite{LiuWooley} that
\begin{equation}\label{power}
  E(N)\ll N^{3/8+\epsilon}
\end{equation}
for any fixed $\epsilon>0$. Subsequently, Harman and the first author \cite{hk1,hk2} adapted Harman's alternative sieve method \cite{Har83,Har96,PDS} to further improve \eqref{power}. In particular, they proved \cite{hk2} the sharpest bound for $E(N)$ to date:
\begin{equation}\label{power2}
  E(N)\ll N^{7/20+\epsilon}
\end{equation}
for any fixed $\epsilon > 0$. The purpose of this paper is to improve on \eqref{power2} by establishing the following result.

\begin{theorem}\label{theorem}
  One has
  \begin{equation}\label{main}
    E(N)\ll N^{11/32}.
  \end{equation}
\end{theorem}

The improvement in our theorem has two sources. First, we simplify significantly the treatment of the major arcs in the application of the circle method. That removes a barrier to the sieve method that was artificially imposed in \cite{hk2} to avoid certain technical difficulties on the major arcs. By itself, this idea allows us to ``squeeze'' a little more out of the sieve method in \cite{hk2} and to reduce the exponent $7/20$ in \eqref{power2} to approximately $0.347$. Our second innovation is a new bound for a triple exponential sum, Lemma~\ref{lemma32} below, which allows us to strengthen some of the sieve estimates in \cite{hk2}. The stronger sieve is responsible for the further reduction of the exponent in \eqref{main} to $11/32$. We remark that, in contrast to earlier work, the exponent $11/32$ in \eqref{main} is only a convenient approximation to the best possible exponent. In fact, in order to have an $\epsilon$-free bound, we establish a slightly stronger result with exponent $11/32-10^{-4}+\epsilon$, whereas the actual limit of the method is the exponent $11/32-\eta+\epsilon$ for some $\eta \approx 3.6 \times 10^{-4}$. 

The history of the above problem is intertwined with that of the companion question about sums of three squares of primes, and results often come in pairs. Indeed, in \cite{hk1, hk2}, Harman and the first author obtained simultaneously bounds for $E(N)$ and for the related quantity $E_3(N)$, which counts the integers $n \le N$, with $n \equiv 3 \pmod{24}$ and $5 \nmid n$, that cannot be expressed as a sum of three squares of primes. Based on such history and on ``conventional wisdom'' about the circle method, the informed reader may expect that, together with \eqref{main}, we should be able to establish also the bound 
\[
  E_3(N) \ll N^{27/32}.
\]
That, however, is not the case. It is true that our minor arc estimates can be adapted for the proof of such a result, but our treatment of the major arcs relies on the presence of four variables in the problem and does not extend to the ternary problem. Thus, in that problem, we still face the same artificial barrier as in the first author's work with Harman.   

\begin{notation}
  Throughout the paper, the letter $\epsilon$ denotes a sufficiently small positive real number. Any statement in which $\epsilon$ occurs holds for each fixed $\epsilon > 0$, and any implied constant in such a statement is allowed to depend on $\epsilon$. The letter $p$, with or without subscripts, is reserved for prime numbers; $c$ denotes an absolute constant, not necessarily the same in all occurrences. As usual in number theory, $\mu(n)$, $\phi(n)$ and $\tau(n)$ denote, respectively, the M\"obius function, the Euler totient function and the number of divisors function. Also, if $n \in \mathbb N$ and $z \ge 2$, we define
  \begin{equation}\label{eq1.8}
    \psi(n,z)= \begin{cases}
      1 & \text{if $n$ is divisible by no prime $p < z$,} \\
      0 & \text{otherwise.}
    \end{cases}
  \end{equation}
  It is also convenient to extend the function $\psi(n, z)$ to all real $n \ge 1$ by setting $\psi(n,z) = 0$ for $n \notin \mathbb Z$. We write $e(x)=\exp(2\pi ix)$, $e_q(x) = e(x/q)$, and $(a,b)=\gcd(a,b)$, and we use $m \sim M$ as an abbreviation for the condition $M< m \le 2M$. 
\end{notation}

\section{Outline of the proof}
\label{sec.outline}

The theorem will follow by a standard dyadic argument, if we show that
\begin{equation}\label{main2}
  |\mathcal E \cap (N/2,N]| \ll N^{11/32}
\end{equation}
for all sufficiently large $N$. Thus, we fix a large $N$ and define 
\[
  P=\frac 23N^{1/2}, \quad L=\log P, \quad \mathcal{I}= [P/2, P).
\]
We shall construct functions $\rho_j$, $1 \le j\le 3$, such that
\begin{equation}\label{psiine}
  \psi(m, P^{1/2})\psi(k, P^{1/2}) \ge \rho_1(m)\psi(k, P^{1/2}) - \rho_3(m)\rho_2(k),
\end{equation}
where $\psi(m,z)$ is defined by \eqref{eq1.8}. Note that for integers $m \in \mathcal{I}$, $m$ is prime if and only if $\psi(m,P^{1/2})=1$. Therefore, when $n \in \mathcal A \cap (N/2, N]$, \eqref{psiine} yields
\begin{equation}\label{S12}
  \sum_{\substack{p_1^2 +p_2^2 +p_3^2 +p_4^2= n \\ p_j \in \mathcal{I}}} 1
  \ge S_1 - S_2,
\end{equation}
where
\[
\begin{split}
S_1 &= \sum_{\substack{m_1^2 +p_2^2 +p_3^2 +p_4^2 = n
\\ m_1,p_2,p_3,p_4 \in \mathcal{I}}} \rho_1(m_1), \\
S_2 &= \sum_{\substack{m_1^2 +m_2^2 +p_3^2 +p_4^2 = n
\\ m_1,m_2,p_3,p_4 \in \mathcal{I}}} \rho_3(m_1)\rho_2(m_2).
\end{split}
\]
We study $S_1$ and $S_2$ by the circle method. 

Let $\rho_0$ denote the characteristic function of the set of primes. For $0\le j\le 3$, we define
\begin{align}\label{deff}
  f_j(\alpha)=\sum_{m\in \mathcal{I}}\rho_j(m)e(m^2\alpha).
\end{align}
By orthogonality,
\begin{align}\label{orth}
  \sum_{\substack{m_1^2 +m_2^2 +p_3^2 +p_4^2 = n\\ m_1,m_2,p_3,p_4 \in \mathcal{I}}} \rho_j(m_1)\rho_k(m_2)
  =\int_{0}^1 f_j(\alpha)f_k(\alpha)f_0(\alpha)^2e(-n\alpha) \, d\alpha.
\end{align}
The evaluation of the integral on the right side of \eqref{orth} uses that the sieve weights $\rho_j$, $1 \le j \le 3$, have properties that are somewhat similar to the properties of the indicator function of the primes. In particular, our construction in \S\ref{sec.sieve} will yield functions $\rho_j$ with the following three properties:
\begin{enumerate}
  \item [(i)] If $m \in \mathcal I$, one has $\rho_j(m)=0$ unless $\psi(m,P^{0.06})=1$.
  \item [(ii)] Let $A,B > 0$ be fixed. For any non-principal Dirichlet character $\chi$ modulo $q \le L^B$ and for any $u,v \in \mathcal{I}$, one has
  \[
    \sum_{u < m \le v}\rho_j(m)\chi(m) \ll PL^{-A}.
  \]
  \item [(iii)] Let $A > 0$ be fixed. There exist smooth functions $\varrho_j$ and constants $C_j$ such that, for any $u,v \in \mathcal{I}$, one has
  \begin{align*}
    \sum_{u < m \le v}\rho_j(m) &= \sum_{u < m \le v}\varrho_j(m)+O(PL^{-A}) \\ 
    &= C_j(v-u)L^{-1}+O(PL^{-2}).
  \end{align*}
\end{enumerate} 
We remark that these properties are well-known in the case $j = 0$ (the indicator function of the primes): (ii) is then a form of the Siegel--Walfisz theorem, whereas (iii) with $C_0 = 1$ and $\rho_0(m) = (\log m)^{-1}$ is the Prime Number Theorem with a rather weak error term. 

For $1 \le Q \le P$, we introduce the collection of major arcs
\begin{align}
  \mathfrak{M}(Q)=\bigcup_{q\le Q}\bigcup_{\substack{a=1\\ (a,q)=1}}^q 
  \left[ \frac{a}{q}-\frac{Q}{qP^2}, \frac{a}{q}+\frac{Q}{qP^2}\right].
\end{align}
To apply the circle method to the right side of \eqref{orth}, we dissect the unit interval into sets of major and minor arcs, defined as
\begin{align}\label{majorminor}
  \mathfrak{M}=\mathfrak{M}\big( P^{0.01} \big) \quad \textrm{ and } \quad  
  \mathfrak{m}= \big[ P^{-1.99},1+P^{-1.99}\big]\setminus \mathfrak{M}.
\end{align}
We remark that this choice of major and minor arcs differs from those made by earlier authors, who required significantly larger sets of major arcs (e.g., the major arcs in \cite{hk2} are given by $\mathfrak M = \mathfrak M(P^{0.3-\epsilon})$). The modest size of our set of major arcs allows us to use standard techniques from \cite{hk1, Liu} to estimate the contribution of $\mathfrak M$ to the right side of \eqref{orth}. In~\S\ref{majorarcs}, we show that if $\rho_j$ and $\rho_k$ satisfy hypotheses (i)--(iii) above, plus another technical hypothesis, then
\begin{align}\label{lemma51}
  \int_{\mathfrak{M}} f_j(\alpha)f_k(\alpha)f_0(\alpha)^2e(-n\alpha) \, d\alpha
  = (C_jC_k+o(1)) \mathfrak{S}(n)\mathfrak{I}(n/N)NL^{-4}.
\end{align}
Here, $\mathfrak{S}(n)$ and $\mathfrak{I}(t)$ are, respectively, the singular series and the singular integral of the problem, defined by 
\begin{gather*}
  \mathfrak{S}(n) = \sum_{q=1}^\infty \frac{1}{\phi^4(q)}\sum_{\substack{a=1 \\ (a,q)=1}}^q 
  \bigg( \sum_{\substack{r=1 \\ (r,q)=1}}^q e_q\big( ar^2 \big) \bigg)^4 e_q(-an), \\
  \mathfrak{I}(t) = \int_{-\infty}^\infty \bigg(\int_{1/3}^{2/3}e(x^2\gamma) \, dx \bigg)^4 
  e\big(-t\gamma\big) \, d\gamma.
\end{gather*}

To estimate the contribution from the minor arcs, we employ an auxiliary decomposition of the unit interval:
\[
  \mathfrak N = \mathfrak M\big( P^{2/3} \big), \quad 
  \mathfrak n = \big[ P^{-4/3}, 1+P^{-4/3} \big] \setminus \mathfrak N.
\]
Suppose that the sieve weights are constructed so that the constants $C_j$ in (iii) above satisfy 
\begin{equation}\label{constant2}
  C_1 - C_3C_2 > 0,
\end{equation}
and that for some $\sigma$, $3/20 < \sigma < 1/6$, we have
\begin{align}\label{lemma52}
  \sup_{\alpha \in \mathfrak n} |f_j(\alpha)| \ll P^{1-\sigma+\epsilon} \qquad (j=1,2).
\end{align}
In \S\ref{minorarcs}, we show that \eqref{orth}--\eqref{lemma52} yield the bound
\[
  S_1 - S_2 \gg NL^{-4}
\] 
for all but $O(N^{1/2 - \sigma + \epsilon})$ values of $n \in \mathcal A \cap (N/2,N]$. To complete the proof of the theorem, we show in \S\ref{minorarcs} that the sieve construction in \S\ref{sec.sieve} yields weights that satisfy both \eqref{constant2} and \eqref{lemma52} with $\sigma = 5/32 + 10^{-4}$.

\section{Exponential sum estimates}
\label{sec.expsums}

In this section, we collect the exponential sum estimates needed on the minor arcs. In particular, we establish a new estimate for certain triple sums---Lemma \ref{lemma32} below---that is likely to find applications beyond the proof of our main result. In all results, the set $\mathfrak m_\sigma$ is the set of minor arcs defined by
\[
  \mathfrak m_{\sigma} = \big[ QX^{-2}, 1+QX^{-2} \big] \setminus \mathfrak M(Q), \qquad Q = X^{4\sigma}.
\]
In particular, our lemmas apply to any $\alpha$ that appears on the left side of \eqref{lemma52}.

\begin{lemma}\label{lemma31}
  Let $0 < \sigma < 1/6$, $\alpha \in \mathfrak m_\sigma$, and let $\xi_r$ be complex numbers with $|\xi_r| \ll r^{\epsilon}$. Then
  \[
    \sum_{r \sim R} \sum_{rm \sim X} \xi_r e(\alpha r^2m^2) \ll X^{1 - \sigma + \epsilon},
  \]
  provided that $R \ll X^{1-3\sigma}$.
\end{lemma}

\begin{proof}
  This bound is established by Harman \cite{Harnew}. In particular, as a part of its proof, Harman shows that if $\alpha \in \mathfrak m_\sigma$ and $U \le X^{2\sigma} \le R \le X^{1-3\sigma}$, then
  \begin{equation}\label{3eq1}
    \#\big\{ (r,u) \in \mathbb Z^2: r \sim R, \; u \sim U, \; \| \alpha ur^2 \| < R^2X^{2\sigma-2} \big\} \ll RU^{1/2}X^{-\sigma+\epsilon}. \qedhere
  \end{equation}
\end{proof}

\begin{lemma}\label{lemma32}
  Let $0 < \sigma < 1/6$, $\alpha \in \mathfrak m_\sigma$, and let $\xi_{r,s}$ be complex numbers with $|\xi_{r,s}| \ll (rs)^{\epsilon}$. Then
  \[
    \Sigma = \sum_{r \sim R} \sum_{s \sim S} \sum_{rsm \sim X} \xi_{r,s} e(\alpha r^2s^2m^2) \ll X^{1 - \sigma + \epsilon},
  \]
  provided that $R \ll X^{1-3\sigma}$ and $RS^2 \le 0.1X^{1-2\sigma}$.
\end{lemma}

\begin{proof}
  We may assume that $RS \ge X^{1-3\sigma}$, for otherwise the result follows from Lemma \ref{lemma31}. Note that together with the hypothesis $RS^2 \ll X^{1-2\sigma}$, this assumption yields $S \ll X^{\sigma}$ and $R \gg X^{1-4\sigma} \gg X^{2\sigma}$. When $RS \ll X^{1-2\sigma}$, standard estimates for the inner sum (Lemma 2.4 and Theorem 4.1 in \cite{Vaughan}) yield
  \[
    \Sigma \ll \sum_{(r,s) \in \mathcal S} \frac {u^{-1/2}X^{1+\epsilon}/(RS)}{1 + (X/RS)^2|\alpha r^2s^2 - b/u|} + X^{1-\sigma+\epsilon},
  \]
  where $\mathcal S$ denotes the set of pairs $(r,s) \in \mathbb Z^2$ with $r \sim R$, $s \sim S$, for which there exist integers $b, u$ with
  \begin{equation}\label{3eq2}
    1 \le u \le X^{2\sigma}, \quad (b, u) = 1, \quad |\alpha ur^2s^2 - b| < (RS)^2X^{2\sigma-2}.
  \end{equation}
  Suppose that $(r,s) \in \mathcal S$. By Dirichlet's theorem on Diophantine approximations, there exist integers $b_1, u_1$ such that
  \begin{equation}\label{3eq3}
    1 \le u_1 \le 10S^2X^{2\sigma} , \quad (b_1, u_1) = 1, \quad |\alpha u_1r^2 - b_1| < 0.1S^{-2}X^{-2\sigma}.
  \end{equation}
  Combining \eqref{3eq2}, \eqref{3eq3} and the hypothesis $RS^2 \le 0.1X^{1-2\sigma}$, we get
  \[
    |b_1us^2 - bu_1| < 0.1u(2S)^2S^{-2}X^{-2\sigma} + 10S^2X^{2\sigma}(RS)^2X^{2\sigma-2} \le 0.5,
  \]
  whence
  \[
    \frac bu = \frac {b_1s^2}{u_1}, \qquad u = \frac {u_1}{(u_1,s^2)}.
  \]
  Thus,
  \begin{align*}
    \Sigma &\ll \sum_{r \sim R} \frac {u_1^{-1/2}X^{1+\epsilon}/(RS)}{1 + (X/R)^2|\alpha r_1^2 - b_1/u_1|}\sum_{s \sim S} (u_1,s^2)^{1/2} + X^{1-\sigma+\epsilon} \\
    &\ll \sum_{r \sim R} \frac {u_1^{-1/2}X^{1+\epsilon}/R}{1 + (X/R)^2|\alpha r_1^2 - b_1/u_1|} + X^{1-\sigma+\epsilon},
  \end{align*}
  on using standard divisor estimates (see Lemma 2.3 in \cite{KW}). If either $u_1 \ge X^{2\sigma}$ or $|\alpha u_1r^2 - b_1| \ge R^2X^{2\sigma-2}$ this yields the desired bound. Otherwise, we have
  \begin{align*}
    \Sigma &\ll X^{1+\epsilon}U^{-1/2}R^{-1}|\mathcal R| + X^{1-\sigma+\epsilon},
  \end{align*}
  where $\mathcal R$ is the set of integers $r \sim R$ for which there exists an integer $u_1 \sim U$, $1 \le U \le X^{2\sigma}$, such that $\| \alpha u_1r^2 \| < R^2X^{2\sigma - 2}$. Recalling that $X^{2\sigma} \ll R \ll X^{1-3\sigma}$, we see that the desired bound then follows from \eqref{3eq1}.
\end{proof}

\begin{lemma}\label{lemma33}
  Let $0 < \sigma < 1/6$, $\alpha \in \mathfrak m_\sigma$, and let $\xi_r, \eta_s$ be complex numbers with $|\xi_r| \ll r^{\epsilon}$, $|\eta_s| \ll s^{\epsilon}$. Then
  \[
    \sum_{r \sim R} \sum_{rs \sim X} \xi_r\eta_s e(\alpha r^2s^2) \ll X^{1 - \sigma + \epsilon},
  \]
  provided that $X^{2\sigma} \ll R \ll X^{1-4\sigma}$.
\end{lemma}

This is a classical bound due to Ghosh \cite{Ghosh}.

\begin{lemma}\label{lemma34}
  Let $0 < \sigma < 1/6$, $\alpha \in \mathfrak m_\sigma$, and let $\xi_r, \eta_s$ be complex numbers with $|\xi_r| \ll r^{\epsilon}$, $|\eta_s| \ll s^{\epsilon}$. Then
  \[
    \Sigma = \sum_{r \sim R} \sum_{s \sim S} \sum_{rsm \sim X} \xi_r\eta_s \psi(m, z) e(\alpha r^2s^2m^2) \ll X^{1 - \sigma + \epsilon},
  \]
  provided that $R \le X^{2\sigma}$, $S \le X^{2\sigma}$, $RS \ll X^{1-3\sigma}$, and $z \le X^{1-6\sigma}$.
\end{lemma}

\begin{proof}
  Let $\Pi = \prod_{p < z} p$. We have
  \[
    \Sigma = \sum_{r \sim R} \sum_{s \sim S} \sum_{d \mid \Pi} \sum_{drsm \sim X} \xi_r\eta_s\mu(d) e(\alpha r^2s^2m^2d^2),
  \]
  where $\mu$ is the M\"obius function. We break $\Sigma$ into several subsums depending on the relative sizes of $d, R, S$. \\

  {\sl Case 1:} $dRS \le X^{1-3\sigma}$. The corresponding terms of $\Sigma$ form a Type I sum that can be estimated using Lemma \ref{lemma31}. \\

  {\sl Case 2:} $dRS > X^{4\sigma}$, or $dR > X^{2\sigma}$, or $dS > X^{2\sigma}$. Then we can use the argument in Harman \cite[Theorem 3.1]{PDS} to split the corresponding terms of $\Sigma$ into $\ll (\log X)^2$ subsums of Type II that can be estimated using Lemma \ref{lemma33}. For example, when $R \le X^{2\sigma} < dR$ and $d \mid \Pi$, $d$ can be factored as $d = d_1d_2$ so that $X^{2\sigma} \ll d_1R \ll X^{1-4\sigma}$. \\

  {\sl Case 3:} $dR \le X^{2\sigma}$, $dS \le X^{2\sigma}$, $X^{1-3\sigma} < dRS \le X^{4\sigma}$. Then $\Sigma$ can be split into $\ll \log N$ sums of the form in Lemma \ref{lemma32} with $(r,s) = (rs,d)$. Indeed, we have
\[
  RS \ll X^{1-3\sigma}, \quad RSd^2 \le (dR)(dS) \le X^{4\sigma} \le 0.1X^{1-2\sigma}. \qedhere
\]
\end{proof}

\section{Sieve construction}
\label{sec.sieve}

In this section, we present our sieve construction, which has a lot in common with the one used in \cite{hk2} by Harman and the first author. We construct arithmetic functions $g_1, g_2, b_1, b_2, b_3$ such that
\begin{gather}
  \label{eq0} \psi(m, P^{1/2}) = g_1(m) - b_1(m) + b_2(m), \\
  \label{eq0a} \psi(m, P^{1/2}) = g_2(m) - b_3(m),
\end{gather}
where $b_i(m) \ge 0$ and we can apply Lemmas \ref{lemma33} and \ref{lemma34} to estimate the exponential sums $\sum_m g_i(m)e(\alpha m^2)$. Our decompositions are based on Buchstab's identity
\begin{equation}\label{eq1}
  \psi(m,z_1) = \psi(m,z_2) - \sum_{z_2 \le p < z_1} \psi(m/p, p) \qquad (2 \le z_2 < z_1).
\end{equation}

For $3/20 < \sigma < 1/6$, put
\[
  z = P^{1-6\sigma}, \quad V = P^{2\sigma}, \quad W = P^{1-4\sigma}, \quad Y = P^{1-3\sigma}.
\]
The reader will recognize these quantities as the various limits on the sizes of the summation variables in the exponential sum bounds from \S\ref{sec.expsums}. We treat $\sigma$ as a numerical parameter to be chosen later in the proof of our theorem; its value will eventually be set to $\sigma = 5/32 + 10^{-4}$.
  
We first describe the identity \eqref{eq0}. By \eqref{eq1},
\begin{align}\label{eq2}
  \psi(m, P^{1/2}) &= \psi(m,z) - \bigg\{ \sum_{z \le p < V} + \sum_{V \le p \le W} + \sum_{W < p < P^{1/2}} \bigg\} \psi(m/p, p) \notag\\
                   &= \psi_1(m) - \psi_2(m) - \psi_3(m) - \psi_4(m), \quad \text{say}. 
\end{align}
In this decomposition, $\psi_1$ and $\psi_3$ will contribute to $g_1$ and $\psi_4$ will be a part of $b_1$; we decompose $\psi_2$ further. Another application of Buchstab's identity gives
\begin{align}\label{eq3}
  \psi_2(m) &=  \sum_{z \le p_1 < V} \bigg\{ \psi(m/p_1, z) - \sum_{z \le p_2 < p_1 < V} \psi(m/(p_1p_2), p_2) \bigg\} \notag\\
            &= \psi_5(m) - \psi_6(m), \quad \text{say}. 
\end{align}
We now write
\begin{equation}\label{eq4}
  \psi_6(m) = \psi_7(m) + \dots + \psi_{10}(m),
\end{equation}
where $\psi_i$ is the part of $\psi_6$ subject to the following extra conditions on the product $pq$:
\begin{itemize}
  \item $\psi_7(m)$: $p_1p_2 < V$;
  \item $\psi_8(m)$: $V \le p_1p_2 \le W$;
  \item $\psi_9(m)$: $W < p_1p_2 \le Y$;
  \item $\psi_{10}(m)$: $p_1p_2 > Y$.
\end{itemize}
In our final decomposition, $\psi_5$ and $\psi_8$ contribute to $g_1$ and $\psi_{10}$ contributes to $b_2$; we give further decompositions of $\psi_7$ and $\psi_9$.

We apply \eqref{eq1} twice more to $\psi_7$:
\begin{align}\label{eq4a}
  \psi_7(m) &= \sum_{p_1,p_2} \bigg\{ \psi(m/(p_1p_2),z) - \sum_{z \le p_3 < p_2} \psi(m/(p_1p_2p_3), z) \notag \\
  &\qquad \qquad \qquad \quad + \sum_{z \le p_4 < p_3 < p_2} \psi(m/(p_1p_2p_3p_4), p_4) \bigg\} \notag \\
  &= \psi_{11}(m) - \psi_{12}(m) + \psi_{13}(m), \quad \text{say}. 
\end{align}
We next apply Buchstab's identity to $\psi_9$ and obtain
\begin{align}\label{eq5}
  \psi_9(m) &= \sum_{p_1, p_2} \bigg\{ \psi(m/(p_1p_2),z) \notag \\
  &\quad - \sum_{z \le p_3 < p_2} \bigg( \sum_{p_1p_2p_3 \le Y} + \sum_{p_1p_2p_3 > Y} \bigg) 
  \psi(m/(p_1p_2p_3), p_3) \bigg\} \notag \\
  &= \psi_{14}(m) - \psi_{15}(m) - \psi_{16}(m), \quad \text{say}. 
\end{align}
Note that the summation conditions in $\psi_{15}$ imply $p_2p_3 \le P^{2/3 - 2\sigma} \le W$. Thus, a final application of \eqref{eq1} yields
\begin{align}\label{eq6}
  \psi_{15}(m) &= \sum_{p_1,p_2,p_3} \bigg\{ \sum_{p_2p_3 \ge V} + \sum_{p_2p_3 < V} \bigg\} \psi(m/(p_1p_2p_3), p_3) \notag \\
  &= \psi_{17}(m) + \sum_{\substack{p_1,p_2,p_3\\ p_2p_3 < V}} \bigg\{ \psi(m/(p_1p_2p_3),z) 
  - \sum_{z \le p_4 < p_3} \psi(m/(p_1 \cdots p_4), p_4) \bigg\} \notag \\
  &= \psi_{17}(m) + \psi_{18}(m) - \psi_{19}(m), \quad \text{say}. 
\end{align}
Finally, we split $\psi_{13}$, $\psi_{16}$, and $\psi_{19}$ into ``good'' and ``bad'' parts, which we denote $\psi_j^g$ and $\psi_j^b$, respectively. We collect in $\psi_j^g$ the terms in $\psi_j$ in which a subproduct of $p_1p_2p_3p_4$ lies within the ranges $[V,W]$ or $[P/W, P/V]$; they will contribute to $g_1$. The remaining terms in $\psi_j$ are placed in $\psi_j^b$ and will contribute to $b_1$ or $b_2$, depending on the value of $j$.

Combining \eqref{eq2}--\eqref{eq6}, we now have \eqref{eq0} with
\begin{gather*}
  \begin{split}
    g_1(m) = \psi_1(m) &- \psi_3(m) - \psi_5(m) + \psi_8(m) + \psi_{11}(m) - \psi_{12}(m) + \psi_{13}^g(m) \\
                     &+ \psi_{14}(m) - \psi_{16}^g(m) - \psi_{17}(m) - \psi_{18}(m) + \psi_{19}^g(m), \\
  \end{split} \\
  b_1(m) = \psi_4(m)+\psi_{16}^b(m), \qquad b_2(m) = \psi_{10}(m) + \psi_{13}^b(m) + \psi_{19}^b(m).
\end{gather*}
We remark that each term $\psi_j$ that appears in $g_1$ leads to an exponential sum that can be estimated using Lemmas \ref{lemma33} or \ref{lemma34}, and that each term $\psi_j^g(m)$ leads to a sum that can be estimated using Lemma \ref{lemma33}.

We now turn to \eqref{eq0a}. We have
\begin{align}\label{eq7}
  \psi_2(m) &= \bigg\{ \sum_{z \le p \le Y^{1/2}} + \sum_{Y^{1/2} < p < V} \bigg\} \psi(m/p,p) = \psi_{20}(m) + \psi_{21}(m), \quad \text{say}.
\end{align}
The term $\psi_{21}$ will contribute to $b_3$; we apply \eqref{eq1} twice to $\psi_{20}$. That gives
\begin{align}\label{eq8}
  \psi_{20}(m) &= \sum_{z \le p_1 \le Y^{1/2}} \bigg\{ \psi(m/p_1,z) - \sum_{z \le p_2 < p_1} \psi(m/(p_1p_2), z)\notag \\
  &\qquad \qquad \qquad \qquad  + \sum_{z \le p_3 < p_2 < p_1} \psi(m/(p_1p_2p_3), p_3) \bigg\} \notag \\
                   &= \psi_{22}(m) - \psi_{23}(m) + \psi_{24}(m), \quad \text{say}. 
\end{align}
We split $\psi_{24}$ into ``good'' and a ``bad'' parts, and then further split $\psi_{24}^b(m)$ in two:
\begin{align}\label{eq9}
  \psi_{24}^b(m) &= \sum_{p_1,p_2,p_3} \bigg\{ \sum_{p_1p_2p_3^2 \le Y} + \sum_{p_1p_2p_3^2 > Y} \bigg\} \psi(m/(p_1p_2p_3), p_3) \notag\\
  &= \psi_{25}(m) + \psi_{26}(m), \quad \text{say}. 
\end{align}
We apply Buchstab's identity two more times to $\psi_{25}$:
\begin{align}\label{eq10}
  \psi_{25}(m) &= \sum_{p_1,p_2,p_3} \bigg\{ \psi(m/(p_1p_2p_3), z) - \sum_{z \le p_4 < p_3} \psi(m/(p_1 \cdots p_4), z) \notag\\
  & \qquad\qquad\qquad\qquad\qquad + \sum_{z \le p_5 < p_4 < p_3} \psi(m/(p_1 \cdots p_5), p_5) \bigg\} \notag \\
               &= \psi_{27}(m) - \psi_{28}(m) + \psi_{29}(m), \quad \text{say}. 
\end{align}
Finally, we split $\psi_{26}$ and $\psi_{29}$ into ``good'' and ``bad'' subsums. We remark that the summation conditions in $\psi_{25}$ imply $p_1p_3 \le W$. (Otherwise, we would have $p_2p_3 \le P^{\sigma}$, whence $p_3 \le P^{\sigma/2}$ and $p_1p_3 \le P^{1/2 - \sigma}$; the latter contradicts the assumption $p_1p_3 > W$ when $\sigma < 1/6$.) Therefore, the exponential sums with coefficients $\psi_{27}$ and $\psi_{28}$ can be estimated either by Lemma \ref{lemma33} (when $p_1p_3 \ge V$) or by Lemma \ref{lemma34} (when $p_1p_3 < V$). Combining \eqref{eq2} and \eqref{eq7}--\eqref{eq10}, we have \eqref{eq0a} with
\begin{gather*}
  \begin{split}
    g_2(m) = \psi_1(m) &- \psi_3(m) - \psi_5(m) - \psi_{22}(m) + \psi_{23}(m) - \psi_{24}^g(m) \\
                       &- \psi_{26}^g(m) - \psi_{27}(m) + \psi_{28}(m) - \psi_{29}^g(m), \\
  \end{split} \\
  b_3(m) = \psi_4(m) + \psi_{21}(m) + \psi_{26}^b(m) + \psi_{29}^b(m).
\end{gather*}

It follows from (\ref{eq0}) and (\ref{eq0a}) that
\[
  \psi(m, P^{1/2})\psi(k, P^{1/2}) \ge g_1(m)\psi(k, P^{1/2}) - b_1(m)g_2(k).
\]
Thus, we may choose the sieve functions $\rho_j$, $1\le j\le 3$, in \eqref{psiine} as
\[
  \rho_1=g_1, \quad \rho_2=g_2 \quad \textrm{and} \quad \rho_3=b_1.
\]
It is clear from the above construction that this choice leads to respective generating functions $f_1$ and $f_2$ that satisfy inequality \eqref{lemma52}. Furthermore, all three functions are supported on integers $m$ with $\psi(m, z) = 1$, so hypothesis (i) in \S\ref{sec.outline} is satisfied as long as $\sigma < 0.1566\dots$.

\section{The proof of Theorem \ref{theorem}}

In this section, we demonstrate that the functions $\rho_1, \rho_2, \rho_3$ above with $\sigma = 5/32 + \delta$, where $\delta > 0$ is a fixed, sufficiently small constant, have all the properties postulated in \S\ref{sec.outline}.

\subsection{The major arcs}
\label{majorarcs}

We first justify the major arc approximation \eqref{lemma51}. As explained above, the sieve weights satisfy hypothesis (i) in \S\ref{sec.outline}, provided that $\delta \le 10^{-4}$, for example. The hypotheses (ii) and (iii) on the distribution of the $\rho_j$'s follow by partial summation from the Prime Number Theorem and from the Siegel--Walfisz theorem in the form given by Iwaniec and Kowalski \cite[(5.79)]{IK}. In particular, the constants $C_j$ in hypothesis (iii) arise as linear combinations of multiple integrals corresponding to the different functions $\psi_j^*$ in \S\ref{sec.sieve}. For example, our choice of $\rho_3$ results in
\[
  C_3 = \log \left( \frac {4\sigma}{1-4\sigma} \right) + \iiint_{D_{16}} \omega\left( \frac {1-u_1-u_2-u_3}{u_3} \right) \, \frac {du_1du_2du_3}{u_1u_2u_3^2},
\] 
where $\omega$ is the so-called Buchstab function from sieve theory and $D_{16}$ is the set in $\mathbb R^3$ defined by the conditions
\begin{gather*}
  1-6\sigma \le u_3 \le u_2 \le u_1 \le 2\sigma, \quad 
  1-4\sigma \le u_1 + u_2 \le 1-3\sigma \le u_1 + u_2 + u_3, \\ 
  \text{no subsum of } u_1+u_2+u_3 \text{ lies in the set } [2\sigma, 1-4\sigma] \cup [4\sigma, 1-2\sigma].
\end{gather*}
The reader will find the definition of $\omega$ and a thorough explanation of the nature of the approximations in (iii) in Harman's monograph \cite[pp. 15--16]{PDS}. A numerical evaluation of the constants $C_j$ reveals that when $\sigma = 5/32 + 10^{-4}$, we have
\[
  C_1 > 1.665, \quad C_2 < 2.096, \quad \text{and} \quad C_3 < 0.769,  
\]
and so \eqref{constant2} holds when $\delta = 10^{-4}$.

Beyond properties (i)--(iii) in \S\ref{sec.outline}, we also need an additional, more technical arithmetic hypothesis on the functions $\rho_j$:
\begin{enumerate}
  \item [(iv)] The function $\rho_j$ can be expressed as a linear combination of $O(L^c)$ bilinear sums of the form
  \[
    \sum_{uv = m} \alpha_u\beta_v,
  \]
  where $|\alpha_u| \le \tau(u)^c$, $|\beta_v| \le \tau(v)^c$, and either $P^{0.06} \le v \le P^{0.94}$ (type II), or $v \ge P^{0.06}$ and $\beta_v = 1$ for all $v$ (type I).
\end{enumerate} 
Note that in the case $j = 0$ (i.e., when $\rho_j$ is the indicator function of the primes), we can obtain such a decomposition by applying Vaughan's or Heath-Brown's combinatorial identities for von Mangoldt's function. Hypothesis (iv) states that our sieve functions can be similarly decomposed. Indeed, with the exception of $\psi_1(m)$, every other arithmetic function $\psi_j^\bullet(m)$ in \S\ref{sec.sieve} can be viewed as a type II sum under hypothesis (iv). Finally, in the notation of Lemma \ref{lemma34}, we have
\[
  \psi_1(m) = \sum_{\substack{ m = dv\\d \mid \Pi}} \mu(d),
\]
and the sum on the right can be split into $O(L)$ subsums, each either of type~II, or of type I with $v \ge P^{0.94}$. 

We next sketch how hypotheses (i)--(iv) lead to a proof of \eqref{lemma51}. The proof of \eqref{lemma51} in the case $j = k = 0$ is by now a standard matter: see for example Liu \cite{Liu}, where he establishes such a result for $\mathfrak M = \mathfrak M(P^{0.4-\epsilon})$. Harman and the first author \cite{hk1, hk2}  showed that the arguments from \cite{Liu} can be applied to more general integrals of the above type, though at the cost of some technical complications. The major inconvenience in those works is the possibility (not present in \cite{Liu}) that when $\alpha$ is on a major arc centered at $a/q$, $(a,q) = 1$, the denominator $q$ need not be relatively prime to all the integers in the support of the sieve weights (see \cite[pp. 6--7]{hk1} and \cite[p. 1974]{hk2}). Our choice of major arcs \eqref{majorminor} and hypothesis (i), however, rule out that possibility in the present context. Therefore, we can follow the argument in \cite{Liu} almost verbatim except for the estimation of the quantity $J(g)$ in \cite[\S3]{Liu}, which we need to replace by
\[
  J(g) = \sum_{r \sim R} [r,g]^{-1+\epsilon}\sideset{}{^*}\sum_{\chi \!\! \!\! \mod r} 
  \max_{|\beta| \le P^{-1.99}} \bigg| \sum_{m \in \mathcal I} \rho_j(m)\chi(m)e( \beta m^2 ) \bigg|,
\]
where $1 \le R \le P^{0.01}$, $g$ is an integer with $1 \le g \le N$, and the middle sum is over all primitive Dirichlet characters $\chi$ modulo $r$. The estimation of this average can be handled using the modification of Liu's argument outlined in \cite[(4.10)--(4.12)]{hk1}. Using hypothesis (iv), we can replace \cite[Lemma 2.1]{Liu} with the inequality (cf. \cite[(4.12)]{hk1})
\begin{equation}\label{meanvalue}
  \sum_{r \sim R} \; \sideset{}{^*}\sum_{\chi \!\! \!\! \mod r} \int_{-T}^T |F_j(1/2 + it, \chi)| \, dt
  \ll L^c\big( P^{1/2} + RT^{1/2}P^{0.47} + R^2T \big),
\end{equation}
where $F_j(s,\chi)$ is the Dirichlet polynomial
\[
  F_j(s, \chi) = \sum_{m \in \mathcal I} \rho_j(m)\chi(m)m^{-s}.
\]
Once we have \eqref{meanvalue} at our disposal, we follow the argument in \cite[p. 8]{hk1} to obtain the needed variants of \cite[Lemmas 3.1 and 3.2]{Liu} and complete the proof of \eqref{lemma51}. 


\subsection{The minor arcs}
\label{minorarcs}

We write $\mathcal E_N = \mathcal E \cap (N/2,N]$, 
\[
  F(\alpha)=f_1(\alpha)f_0(\alpha)-f_3(\alpha)f_2(\alpha), \quad 
  K(\alpha) = \sum_{ n\in \mathcal E_N} e(-\alpha n).
\]
In particular, we have 
\begin{equation}\label{orth2}
  S_1 - S_2 = \int_0^1 F(\alpha)f_0(\alpha)^2e(-\alpha n) \, d\alpha.
\end{equation}
For $n\in (N/2,N]\cap \mathcal{A}$, one has
\begin{align}\label{SnIn}
  \mathfrak{S}(n)\gg 1 \quad \text{and} \quad \mathfrak{I}(n/N)\gg 1.
\end{align}
From \eqref{S12}, \eqref{lemma51}, \eqref{constant2}, \eqref{orth2}, and \eqref{SnIn}, we deduce that 
\[
  -\int_{\mathfrak m} F(\alpha)f_0(\alpha)^2e(-n\alpha) \, d\alpha \gg NL^{-4}
\]
for all $n \in \mathcal E_N$. Summing these inequalities over $n$, we obtain
\begin{align}\label{simplebound} 
  |\mathcal{E}_N|NL^{-4} \ll \bigg| \int_{\mathfrak{m}}F(\alpha)f_0(\alpha)^2K(\alpha) \, d\alpha \bigg|.
\end{align}

Recall that by the construction of $\rho_1$ and $\rho_2$, we can use Lemmas \ref{lemma33} and~\ref{lemma34} to establish \eqref{lemma52}. Thus, we obtain from \eqref{lemma52} that
\begin{align}\label{J12} 
  \int_{\mathfrak n} F(\alpha)f_0^2(\alpha)K(\alpha) \, d\alpha \ll P^{1-\sigma+\epsilon}(I_1+I_2),
\end{align}where
\[
  I_1= \int_0^1 \big|f_0^3(\alpha)K(\alpha)\big| \, d\alpha, \quad 
  I_2= \int_0^1 \big| f_3(\alpha)f_0^2(\alpha)K(\alpha) \big| \, d\alpha.
\]

We now define a function $\Delta$ on $\mathfrak{N}$ by 
\begin{align*}
  \Delta(\alpha)=(q+N|q\alpha-a|)^{-1}
\end{align*}
when $|q\alpha - a| \le P^{-4/3}$, with $1\le a\le q\le P^{2/3}$ and $(a,q)=1$. By the main result in Ren \cite{ren}, when $\alpha \in \mathfrak N$, we have 
\begin{align}\label{boundf}
  f_0(\alpha)\ll P^{1+\epsilon}\Delta(\alpha)^{1/2}+P^{5/6+\epsilon}. 
\end{align}
From \eqref{boundf}, we deduce that
\begin{align} \label{prun} 
  \int_{\mathfrak{m} \cap \mathfrak N} F(\alpha)f_0(\alpha)^2K(\alpha) \, d\alpha
  \ll P^{5/6+\epsilon}I_3 + P^{1+\epsilon}I_4,
\end{align}
where
\begin{gather*}  
  I_3 = \int_0^1 \big|F(\alpha)f_0(\alpha)K(\alpha) \big| \, d\alpha, \\
  I_4 = \int_{\mathfrak m \cap \mathfrak N} \big| F(\alpha)f_0(\alpha)\Delta(\alpha)^{1/2}K(\alpha) \big| \, d\alpha.
\end{gather*}

We can estimate $I_1, I_2$ and $I_3$ similarly to Wooley \cite[(3.21)--(3.23)]{Wool}. This yields the bounds
\begin{align}\label{I1}  
  I_1, \, I_2 ,\, I_3 \ll N^{3/4+\epsilon}|\mathcal{E}_N|^{1/2}+N^{1/2+\epsilon}|\mathcal{E}_N|.
\end{align}
Moreover, an argument similar to that in Wooley \cite[(3.27)--(3.29)]{Wool} gives
\begin{align}\label{I2} 
  I_4 \ll P^{1+\epsilon}|\mathcal{E}_N|^{3/4}+P^{1+\epsilon}Q^{-1/2}|\mathcal{E}_N|,
\end{align}
where $Q = P^{0.01}$. We conclude from \eqref{J12} and \eqref{prun}--\eqref{I2} that
\begin{align} \label{upper} 
  \int_{\mathfrak m} F(\alpha)f_0(\alpha)^2K(\alpha) \, d\alpha
  &\ll  N^{5/4-\sigma/2+\epsilon}|\mathcal{E}_N|^{1/2}+N^{1+\epsilon}|\mathcal{E}_N|^{3/4}
  + N^{0.998}|\mathcal{E}_N| \notag\\
  &\ll  N^{5/4-\sigma/2+\epsilon}|\mathcal{E}_N|^{1/2} + N^{0.998}|\mathcal{E}_N|.
\end{align}
Finally, combining \eqref{simplebound} and \eqref{upper} and recalling that $\sigma = 5/32+10^{-4}$, we obtain
\[
  |\mathcal{E}_N| \ll N^{1/2-\sigma+\epsilon} \ll N^{11/32}.
\]
This establishes \eqref{main2} and completes the proof of the theorem. \\

{\it Acknowledgment.} This collaboration originated during the workshop on Analytic Number Theory at Oberwolfach, October 20--26, 2013. The authors would like to thank the Mathematics Institute and the organizers of that meeting for their hospitality and the stimulating working environment. Moreover, A. Kumchev wants to express his gratitude to the Morningside Center for Mathematics at the Chinese Academy of Sciences for hospitality during the Workshop on Number Theory, July 20--28, 2014, when work on this project was completed.

\end{document}